\documentclass[11pt]{amsart}
\usepackage{amsmath,amssymb,amsthm,booktabs}
\usepackage[margin=1.1in]{geometry}
\usepackage[hidelinks]{hyperref}

\newtheorem{theorem}{Theorem}
\newtheorem{proposition}[theorem]{Proposition}
\newtheorem{lemma}[theorem]{Lemma}

\theoremstyle{remark}
\newtheorem{remark}[theorem]{Remark}

\newcommand{\lam}{\lambda}
\newcommand{\cP}{\mathcal P}
\newcommand{\cE}{\mathcal E}
\newcommand{\cI}{\mathcal I}
\newcommand{\cR}{\mathcal R}
\newcommand{\Z}{\mathbb Z}
\newcommand{\E}{\mathbb E}

\title{Few small primes suffice to cover unit groups}
\author{Scott D. Hughes}
\thanks{Generative AI tools assisted with the development, preparation, and review of this work.}
\date{}

\begin{document}

\begin{abstract}
\setlength{\emergencystretch}{1em}
Klurman, Shparlinski and Teräväinen showed that there is a set of at most $(\log Q)^{1+\varepsilon}$ primes of
polylogarithmic size whose subset products cover the unit group $(\Z/q\Z)^\times$ for almost all moduli $q\le Q$. We show
that $(1/\log2+o(1))\log Q$ primes suffice, which is optimal to first order, since $m$ primes have at most $2^m$ subset
products. The proof combines a doubling argument of Erd\H{o}s--R\'enyi type for random subset products in a finite abelian
group, modified to tolerate a small set of exceptional characters, with a one-time ``repair'' of the few primitive
characters whose $L$-functions have zeros near $s=1$. An exact weighted repair gives a variant with at most
$(2/\log2+o(1))\log Q$ primes and an exceptional set described explicitly.
\end{abstract}

\maketitle

\section{Introduction}

For a finite set $\cP$ of primes and an integer $q\ge2$, say that $\cP$ \emph{covers} $q$ if
\[
\Bigl\{\prod_{p\in F}p \bmod q : F\subseteq\cP\Bigr\}=(\Z/q\Z)^\times ,
\]
equivalently, every prime in $\cP$ is coprime to $q$ and every unit modulo $q$ is a product of \emph{distinct} elements
of $\cP$ (each prime used at most once). Since
there are only $2^{|\cP|}$ subset products, covering $q$ needs $|\cP|\ge\log_2\varphi(q)$.

Klurman, Shparlinski and Teräväinen \cite[Theorem 1.1]{KST} proved: for every $A>0$ and $\varepsilon>0$ there is a set
$\cP$ of at most $(\log Q)^{1+\varepsilon}$ primes, all at most $(\log Q)^{C(A+1)}$, that covers every $q\le Q$ outside
an exceptional set of size $O_{\varepsilon,A}(Q(\log Q)^{-A})$. Their proof selects random primes and uses a Chernoff
bound together with a union bound over all pairs $(q,\chi)$. Their separation estimate for characters of small conductor
loses a factor $(\log Q)^{-\varepsilon/2}$, and the random selection must give sufficient separation simultaneously for
all these pairs. Here we repair the few exceptional primitive characters before using a two-stage covering argument.
Throughout, $\lam=\log2$, $\ell=\log Q$, and all logarithms are natural.

\begin{theorem}\label{thm:main}
For every $A>0$ there is $C_A=O(A+1)$ such that, for all sufficiently large $Q$, there is a set $\cP_Q$ of distinct primes
with
\[
(\log Q)^{A+3}<p\le(\log Q)^{C_A}\quad(p\in\cP_Q),\qquad
|\cP_Q|\le\frac{\log Q}{\log2}+O_A\Bigl(\frac{\log Q}{\log\log Q}\Bigr),
\]
that covers every integer $2\le q\le Q$ outside a set of $O_A\bigl(Q(\log Q)^{-A}\bigr)$ moduli.
\end{theorem}

The proof bounds three exceptional sets separately (Section~\ref{sec:proof}): moduli divisible by the conductor of
a ``bad'' primitive character with conductor greater than $Y$, moduli divisible by a selected prime, and
$O_A(Q\ell^{-A-10})$ additional moduli left uncovered by the selected outcome. Theorem~\ref{thm:structured} removes
the last set at the cost of a larger leading constant. The description of the first set uses zeros of Dirichlet
$L$-functions and does not assert an efficient membership algorithm.

\begin{theorem}\label{thm:structured}
For every $A>0$ there is $C_A=O(A+1)$ such that, for all sufficiently large $Q$, there is a set $\cP'_Q$ of distinct
primes in $((\log Q)^{A+3},(\log Q)^{C_A}]$ satisfying
\[
|\cP'_Q|\le\frac{2\log Q}{\log2}+O_A\Bigl(\frac{\log Q}{\log\log Q}\Bigr),
\]
that covers every $2\le q\le Q$ with
$q\notin\cE_{\rm an}$ and $\gcd(q,\prod_{p\in\cP'_Q}p)=1$, where $\cE_{\rm an}$ is the set defined in
\eqref{eq:Ean} below, of size $O_A(Q\ell^{-A-9})$.
\end{theorem}

\begin{proposition}\label{prop:opt}
If a set of $m$ primes covers some $q\in[Q/2,Q]$, then $m\ge(1/\log2-o(1))\log Q$. In particular the leading constant
in Theorem~\ref{thm:main} cannot be improved.
\end{proposition}

\subsection*{Method}
Fix a modulus $q$ and adjoin random primes $p_1,p_2,\dots$ one at a time; the set of subset products grows from $S$ to
$S\cup p_jS$. For a fixed finite abelian group, a classical doubling argument going back to Erd\H{o}s and R\'enyi
\cite{ER} shows that about $\log_2|G|$ random elements suffice when every nonprincipal character $\chi$ has small
$|\E\chi(g)|$. For $G=(\Z/q\Z)^\times$ and random primes, $\E\chi(p)$ is small unless the primitive character inducing
$\chi$ has an $L$-function zero close to $s=1$. Zero-density estimates show there are few such characters. Those of large
conductor are handled by discarding their multiples. Those of small conductor are \emph{repaired once}: for each we
insert about $\log\log Q$ primes from a residue class where the character has real part at most $-1/2$
(Lemma~\ref{lem:repair}), at a total cost of $o(\log Q/\log\log Q)$ primes. Lemma~\ref{lem:cover} is a version of the
doubling argument that tolerates such a repaired exceptional set. The only step that ranges over all moduli is linearity
of expectation, so no union bound over characters is needed. For the structured variant, an elementary weighted law
supported on the Boolean products of two primes annihilates each bad character exactly (Lemma~\ref{lem:exact-repair}).
This gives a conditional failure probability $o(1/Q)$ for each eligible modulus, allowing all further random failures
to be removed with the stated leading constant in Theorem~\ref{thm:structured}. We do not claim that this constant is
optimal.

\begin{remark}
The finite covering lemma (Lemma~\ref{lem:cover}) is a routine extension of the Erd\H{o}s--R\'enyi argument; the new
ingredients are the repair step and its analytic bookkeeping. We do not track effectivity. Linnik's theorem is used only in the
weak form \eqref{eq:L}, which is effective (cf.\ \cite[Remark~2.2]{KST}), but we have not checked the constants in the
other inputs and make no claim either way.
\end{remark}

\section{A covering lemma in finite abelian groups}\label{sec:lemma}

Let $G$ be a finite abelian group of order $H\ge2$, written multiplicatively, with dual group $\widehat G$. For a
probability measure $\mu$ on $G$ write $\widehat\mu(\chi)=\sum_{x}\mu(x)\chi(x)$. For $g\in G$ let $\beta_g$ be the
measure $\frac12(\delta_1+\delta_g)$ (``include $g$ or not'').

\begin{lemma}\label{lem:cover}
Let $\mu_0$ be a probability measure on $G$ with support $S_0$. Let $g_1,g_2,\dots$ be independent $G$-valued random
variables with a common law $\nu$, independent of $\mu_0$, and put $\mu_j=\mu_{j-1}*\beta_{g_j}$ and
$S_j=\operatorname{supp}\mu_j=S_{j-1}\cup g_jS_{j-1}$. Let $\mathcal B$ be a set of $B$ nonprincipal characters and let
$0\le\delta\le1$, $\epsilon\ge0$ satisfy
\[
\operatorname{Re}\E\chi(g_1)\le\delta\quad(\chi\ne1,\ \chi\notin\mathcal B),\qquad
|\widehat\mu_0(\chi)|\le\epsilon\quad(\chi\in\mathcal B).
\]
Then for all integers $r,s\ge0$ and every $0<\alpha<1$, with $\rho=(1+\delta)/2$ and
$\beta=\delta+(1-\delta)(B+1)\alpha$,
\begin{equation}\label{eq:cover}
\Pr(S_{r+s}\ne G)\ \le\ \frac{1-\alpha}{\alpha}\Bigl((H-1)\rho^r+B\epsilon^2\Bigr)+\alpha H\beta^s .
\end{equation}
\end{lemma}

\begin{proof}
\emph{Collision step.} $\widehat\mu_r(\chi)=\widehat\mu_0(\chi)\prod_{j\le r}\frac{1+\chi(g_j)}2$, and
$\bigl|\frac{1+\chi(g)}2\bigr|^2=\frac{1+\operatorname{Re}\chi(g)}2$. By independence,
$\E|\widehat\mu_r(\chi)|^2\le|\widehat\mu_0(\chi)|^2\rho^r\le\rho^r$ for $\chi\ne1$ with $\chi\notin\mathcal B$, while
$|\widehat\mu_r(\chi)|\le|\widehat\mu_0(\chi)|\le\epsilon$ for $\chi\in\mathcal B$. By orthogonality,
\[
\E\bigl[H\|\mu_r\|_2^2-1\bigr]=\E\sum_{\chi\ne1}|\widehat\mu_r(\chi)|^2\le(H-1)\rho^r+B\epsilon^2 .
\]
If $S_r$ misses $h$ elements, Cauchy--Schwarz on the support gives $\|\mu_r\|_2^2\ge1/(H-h)$, so
$H\|\mu_r\|_2^2-1\ge h/(H-h)$. Hence, by Markov's inequality, $\Pr(h>\alpha H)\le\frac{1-\alpha}\alpha\bigl((H-1)\rho^r
+B\epsilon^2\bigr)$.

\emph{Hole contraction.} Let $T=G\setminus S$ with $|T|=h$. After adjoining $g$ the missing set is $T\cap gT$. With
$F_T(\chi)=\sum_{x\in T}\chi(x)$, orthogonality gives
\[
\E|T\cap gT|=\sum_{x,y\in T}\nu(xy^{-1})
=\frac1H\sum_{\chi}|F_T(\chi)|^2\,\E\chi(g)
=\frac{h^2}H+\frac1H\sum_{\chi\ne1}|F_T(\chi)|^2\operatorname{Re}\E\chi(g),
\]
the last step because the left side is real and $|F_T(\overline\chi)|=|F_T(\chi)|$. Bounding
$\operatorname{Re}\E\chi(g)$ by $\delta$ for $\chi\ne1$ with $\chi\notin\mathcal B$ and by $1$ for $\chi\in\mathcal B$, and using
$\sum_{\chi\ne1}|F_T(\chi)|^2=Hh-h^2$ and $|F_T(\chi)|\le h$,
\begin{equation}\label{eq:hole}
\E|T\cap gT|\le\delta h+(1-\delta)(B+1)\frac{h^2}H .
\end{equation}
If $h\le\alpha H$ this is at most $\beta h$. Missing sets only shrink, and $g_{j+1}$ is independent of
$g_1,\dots,g_j$. So on the event $\{|G\setminus S_r|\le\alpha H\}$ the expected number of elements missing from $S_{r+s}$
is at most $\alpha H\beta^s$, and Markov's inequality bounds the probability that it is positive. Adding the two bounds
gives \eqref{eq:cover}.
\end{proof}

\begin{remark}[Sharpness of the exceptional term]
If $\nu$ is uniform on a subgroup of index $a$, then exactly $a-1$ nonprincipal characters have $\E\chi(g)=1$ and the
others have $\E\chi(g)=0$. A coset $T$ of the subgroup, $h=H/a$, never shrinks, and \eqref{eq:hole} with $\delta=0$,
$B=a-1$ is an equality. So the exceptional characters cannot simply be ignored in the hole step.
\end{remark}

\section{Repairing a character}

\begin{lemma}\label{lem:repair}
Every nonprincipal Dirichlet character $\psi$ takes a value with real part at most $-1/2$ at some unit $a$. For every prime
$p\equiv a\pmod{\operatorname{cond}\psi}$ one has $|(1+\psi(p))/2|\le1/2$.
\end{lemma}

\begin{proof}
The values of $\psi$ at units form the group of $h$-th roots of unity for some $h\ge2$. If $h=2$ the value $-1$ occurs;
if $h\ge3$ some $h$-th root lies within angle $\pi/h\le\pi/3$ of $-1$ and has real part at most $-\cos(\pi/3)=-1/2$. For
such $p$, either $\psi(p)=0$, in which case $|(1+\psi(p))/2|=1/2$, or $\psi(p)=\psi(a)$ and
$|(1+\psi(p))/2|^2=(1+\operatorname{Re}\psi(p))/2\le1/4$.
\end{proof}

Consequently, if $\mu_0$ is the law of a random subset product of a fixed list containing $t$ distinct primes $p\equiv
a\pmod{\operatorname{cond}\psi}$, and $\chi$ is a character modulo $q$ induced by $\psi$ with $q$ coprime to the list,
then $|\widehat\mu_0(\chi)|=\prod|(1+\chi(p))/2|\le2^{-t}$: the other factors have modulus at most $1$. Since $\chi(p)=
\psi(p)$ for $p\nmid q$, one repair of $\psi$ serves every modulus into which $\psi$ induces, and each primitive character
induces at most one character modulo a given $q$.

For the structured variant we use the following elementary exact cancellation.

\begin{lemma}\label{lem:exact-repair}
Let $\psi$ be a primitive nonprincipal Dirichlet character of conductor $d$, and choose a unit $a$ with
$c:=\operatorname{Re}\psi(a)\le-1/2$. Let $u,v$ be distinct primes with $u\equiv a\pmod d$ and
$v\equiv a^{-1}\pmod d$. Define a random integer $W_\psi$ by
\[
\Pr(W_\psi=1)=\frac{-c}{1-c},\qquad
\Pr(W_\psi=u)=\Pr(W_\psi=v)=\frac{1}{2(1-c)}.
\]
These probabilities are positive and sum to one. If $d\mid q$, $q$ is coprime to $uv$, and $\chi\bmod q$ is induced by
$\psi$, then $\E\chi(W_\psi)=0$. Every value of $W_\psi$ is a Boolean subset product of $\{u,v\}$.
\end{lemma}

\begin{proof}
Put $z=\psi(a)$. Then $|z|=1$, $\psi(a^{-1})=\overline z$ and $-1\le c\le-1/2$. Positivity is immediate, and the sum
of the displayed weights is $(-c+1)/(1-c)=1$. Since $q$ is coprime to $uv$, induction gives $\chi(u)=z$ and
$\chi(v)=\overline z$, so
\[
\E\chi(W_\psi)=\frac{-c}{1-c}+\frac{z+\overline z}{2(1-c)}=0.
\]
The three choices are the empty subset and the two singleton subsets. The doubleton has probability zero; only support
containment in the Boolean product set is needed. If the two residue classes coincide, two distinct primes in that class
still suffice. The weights are exact real numbers, not numerical approximations.
\end{proof}

\section{Analytic input}\label{sec:input}

We use three unconditional facts.

\emph{(L) Linnik's theorem, in a weak quantitative form.} There are absolute constants $L_0\ge1$ and $c>0$ such that
\begin{equation}\label{eq:L}
\pi(x;d,a)\ge c\,\frac{x}{d^{2}\log x}\qquad\text{whenever }d\ge2,\ (a,d)=1,\ x\ge d^{L_0}.
\end{equation}
This follows from \cite[Corollary~18.8]{IK}: for $d$ sufficiently large and $x\ge d^{L}$ one has $\Psi(x;d,a)\gg
x/(\varphi(d)\sqrt d)$ with an absolute implied constant. Removing prime powers costs $O(\sqrt x\log^2x)$, which is
negligible once $x\ge d^{L_0}$ with $L_0\ge\max(L,5)$, and the finitely many small $d$ are covered by the prime number
theorem in progressions after enlarging $L_0$ and shrinking $c$. (See also \cite[Lemma~2.1]{KST}, which states the
stronger bound with $d^{1+\eta}$ for every $\eta>0$, with ineffective constants.)

\emph{(ZD) Zero density for the family of primitive characters} \cite[Theorem~1, (1.8)]{Ju}: for $4/5\le\alpha\le1$,
$T\ge1$ and every $\varepsilon>0$,
\[
\sum_{d\le D}\ \sideset{}{^*}\sum_{\psi\bmod d}N(\alpha,T,\psi)\ll_\varepsilon(D^2T)^{(2+\varepsilon)(1-\alpha)},
\]
where $\sum^*$ runs over primitive characters and $N(\alpha,T,\psi)$ counts zeros of $L(s,\psi)$ in
$\alpha\le\operatorname{Re}s\le1$, $|\operatorname{Im}s|\le T$. Taking $\varepsilon=1$ and $\alpha=1-\sigma$:
\begin{equation}\label{eq:ZD}
\sum_{d\le D}\ \sideset{}{^*}\sum_{\psi\bmod d}N(1-\sigma,T,\psi)\ll(D^2T)^{3\sigma}\qquad(0\le\sigma\le1/5,\ T\ge1,\ D\ge1),
\end{equation}
with an absolute implied constant. (Klurman, Shparlinski and Teräväinen use the same shape with a factor $(\log DT)^{B_0}$,
citing the grand density theorem \cite[Theorem~10.4]{IK}; there the power of the logarithm and the implied constant depend
on $\alpha$. Our parameters define $\sigma$ in terms of $B_0$ (Section~\ref{sec:proof}), so a log power depending on
$\alpha=1-\sigma$ would make that choice circular; \eqref{eq:ZD} has an absolute constant and no logarithm. Any bound of the
form \eqref{eq:ZD} with an extra factor $(\log DT)^{B_0}$, $B_0$ absolute, would do; below we keep a general $B_0\ge0$,
and \eqref{eq:ZD} gives $B_0=0$.)

\emph{(ZF) Prime sums from a zero-free rectangle} \cite[Lemma~2.4]{KST}: for $0<\sigma\le1/7$, $P\ge3$ and a primitive
nonprincipal $\psi$ of conductor at most $Q$ such that $L(s,\psi)$ has no zero in $\operatorname{Re}s\ge1-\sigma$,
$|\operatorname{Im}s|\le P^\sigma$,
\begin{equation}\label{eq:ZF}
\Bigl|\sum_{p\le P}\psi(p)\Bigr|\ll P^{1-\sigma}(\log QP)^3 .
\end{equation}
For completeness we recall the argument of \cite{KST}. Write $\Psi(x;\psi)=\sum_{n\le x}\psi(n)\Lambda(n)$ and
$\Psi(x;d,a)=\sum_{n\le x,\,n\equiv a\,(d)}\Lambda(n)$. With $T=P^\sigma/2$, Perron's formula gives, for $2\le x\le P$,
$\Psi(x;\psi)=-\frac1{2\pi i}\int_{1+1/\log x-iT}^{1+1/\log x+iT}\frac{L'}{L}(s,\psi)\frac{x^s}s\,ds
+O\bigl(x(\log QP)^2/T+\log QP\bigr)$. Put $\eta=1/\log(QP)$ and shift to $a=1-\sigma+\eta$; there are no poles
in between. Since a primitive nonprincipal character has conductor at least $3$, we have $Q\ge3$ and
$\eta<1/2\le T$. Every zero is at distance at least $\eta$ from this contour: either its real part is below $1-\sigma$,
or its imaginary part has absolute value greater than $2T$. Also $a>6/7>5/6$. The local partial-fraction formula
\cite[Lemma~11.1]{MV}, together with the standard bound $O(\log(QP))$ for the number of nontrivial zeros in a fixed-height window,
therefore gives $|L'/L(s,\psi)|\ll(\log QP)^2$ on the contour where $\operatorname{Re}s\le2$.
Where $\operatorname{Re}s>2$ the same bound follows from absolute convergence of the logarithmic-derivative series.
The constants here are absolute and uniform in $0<\sigma\le1/7$.
The vertical integral is bounded using $x^a\le eP^{1-\sigma}$ and
$\int_{-T}^T|a+it|^{-1}\,dt\ll\log(3T)$; each horizontal integral is bounded using
$\int_a^{1+1/\log x}x^u/T\,du\le ex/(T\log x)$. Thus
$|\Psi(x;\psi)|\ll P^{1-\sigma}(\log QP)^3$, uniformly for $2\le x\le P$.
Removing prime powers costs $O(P^{1/2}(\log P)^2)$, which is absorbed because $\sigma\le1/7$.
Finally, if $\vartheta(u;\psi)=\sum_{p\le u}\psi(p)\log p$, partial summation gives
\[
\sum_{p\le P}\psi(p)=\frac{\vartheta(P;\psi)}{\log P}
+\int_2^P\frac{\vartheta(u;\psi)}{u(\log u)^2}\,du.
\]
The integral weight has bounded total mass, proving \eqref{eq:ZF}.

Finally we use the prime number theorem $\pi(x)\sim x/\log x$.

\section{Proof of Theorem~\ref{thm:main}}\label{sec:proof}

\subsection*{Parameters}
Fix $A>0$ and put
\[
K=A+B_0+10,\qquad W=B_0+20,\qquad \sigma=W/C,
\]
with $B_0\ge0$ the absolute constant of \eqref{eq:ZD} (so $B_0=0$ is allowed),
\[
C=\max\{100W(K+W+1),\,2KL_0,\,10(A+2K+W+1)\},
\]
\[
P=\ell^C,\qquad Y=\ell^K,\qquad Z=\ell^{A+3}.
\]
Then $C=O(A+1)$, $C\sigma=W\ge20$, $0<\sigma\le1/(100(K+W+1))<1/7$, and
\begin{equation}\label{eq:params}
\begin{gathered}
\theta:=6K\sigma+3C\sigma^2=\frac{(6K+3W)W}{C}<\frac3{50},\\
C\ge KL_0,\qquad C-2K>A+3,\qquad -K+B_0+\theta<-A-9 .
\end{gathered}
\end{equation}
Call a primitive nonprincipal character \emph{bad} if $L(s,\psi)$ has a zero in $\operatorname{Re}s\ge1-\sigma$,
$|\operatorname{Im}s|\le P^\sigma$, and \emph{good} otherwise.

\subsection*{Few bad characters}
By \eqref{eq:ZD} with $D=Y$, $T=P^\sigma$, the number $B$ of bad characters of conductor at most $Y$ satisfies
\begin{equation}\label{eq:B}
B\ll(Y^2P^\sigma)^{3\sigma}(\log YP^\sigma)^{B_0}=\ell^{\theta}(\log\ell)^{O_A(1)}\le\ell^{1/4}
\end{equation}
for large $Q$. Let
\begin{equation}\label{eq:Ean}
\cE_{\rm an}=\{q\le Q:\ q \text{ is divisible by the conductor } d>Y \text{ of a bad character}\}.
\end{equation}
By \eqref{eq:ZD} with $D=2R$, a dyadic range $(R,2R]$ contains $\ll R^{6\sigma}P^{3\sigma^2}(\log QP)^{B_0}$ bad
conductors, each dividing at most $Q/R$ moduli $q\le Q$. Since $6\sigma<1$ the dyadic sum is dominated by $R\approx Y$,
and since $\log(QP)\le2\ell$,
\[
|\cE_{\rm an}|\ll QY^{-1+6\sigma}P^{3\sigma^2}(\log QP)^{B_0}\ll_A Q\ell^{-K+\theta+B_0}\le Q\ell^{-A-9}
\]
by \eqref{eq:params}.

\subsection*{The repair block}
For each of the $B$ bad characters $\psi$ of conductor $d\le Y$ choose a unit $a_\psi$ as in Lemma~\ref{lem:repair} and
$t=\lceil(A+20)\log_2\ell\rceil$ distinct primes $p\equiv a_\psi\pmod d$ in $(Z,P]$, not used for earlier characters.
This is possible: by \eqref{eq:L} and $P\ge Y^{L_0}\ge d^{L_0}$,
$\pi(P;d,a_\psi)\gg P/(Y^2\log P)\gg_A\ell^{C-2K}/\log\ell$, which exceeds $Z+Bt$ for large $Q$ by \eqref{eq:params}. Let
$\cR$ be the set of all these primes; $m_0:=|\cR|\le Bt=O_A(\ell^{1/4}\log\ell)$. Let $\mu_0$ be the law of a random
subset product of $\cR$. By Lemma~\ref{lem:repair}, for every $q$ coprime to $\prod\cR$ and every character $\chi\bmod q$
induced by a bad character of conductor at most $Y$,
\begin{equation}\label{eq:eps}
|\widehat\mu_0(\chi)|\le2^{-t}\le\epsilon:=\ell^{-A-20}.
\end{equation}

\subsection*{The random primes}
Let $\cI$ be the set of primes in $(Z,P]$ not in $\cR$, and let $p_1,p_2,\dots$ be independent and uniform on $\cI$.
Fix $q\notin\cE_{\rm an}$ coprime to $\prod\cR$. For the finite number $n=r+s$ of draws specified below, put
$U_q=\{p_j\nmid q\text{ for }1\le j\le n\}$. Conditioning on $U_q$ keeps $p_1,\dots,p_n$
independent, now uniform on $\cI_q=\{p\in\cI:p\nmid q\}$. Every nonprincipal $\chi\bmod q$ is induced by a primitive
$\psi$ with $\operatorname{cond}\psi\mid q$; since $q\notin\cE_{\rm an}$, either $\psi$ is good, or $\psi$ is bad of
conductor at most $Y$ and is repaired. For good $\psi$, \eqref{eq:ZF} and the prime number theorem give
$|\sum_{p\le P}\psi(p)|/\pi(P)\ll_A\ell^{-W+3}\log\ell$; removing the primes up to $Z$, the $m_0$ repair primes and the at
most $2\ell$ prime factors of $q$ changes the average by $O_A((Z+m_0+\ell)\log P/P)=O_A(\ell^{A+3-C}\log\ell)$. As
$W\ge20$ and $C-A-3>8$ by the choice of $C$, for large $Q$ and uniformly in $q$,
\begin{equation}\label{eq:delta}
|\E_{p\in\cI_q}\chi(p)|\le\delta:=\ell^{-8}\qquad(\chi\ \text{induced by a good character}).
\end{equation}
There are at most $B$ characters modulo $q$ induced by bad characters of conductor at most $Y$.

\subsection*{Covering one modulus}
Apply Lemma~\ref{lem:cover} to $G=(\Z/q\Z)^\times$ (if $\varphi(q)=1$ there is nothing to prove), $H=\varphi(q)\le Q$,
$\mu_0$ as above and $\nu$ the law of $p\bmod q$ for $p$ uniform on $\cI_q$, with
\[
\alpha=\ell^{-A-6},\qquad r=\Bigl\lceil\frac{\ell+(2A+20)\log\ell}{\lam}\Bigr\rceil,\qquad
s=\Bigl\lceil\frac{\ell+(A+10)\log\ell}{5\log\ell-\log3}\Bigr\rceil .
\]
Since $r\delta=o(1)$, $(H-1)\rho^r\le e^{\ell}(1+\delta)^r2^{-r}\le2\ell^{-2A-20}$, and $B\epsilon^2\le\ell^{-2A-39}$ by
\eqref{eq:B} and \eqref{eq:eps}; so the first term of \eqref{eq:cover} is at most $3\ell^{-A-14}$. Also
$\beta\le\ell^{-8}+(\ell^{1/4}+1)\ell^{-A-6}\le3\ell^{-5}$, and the choice of $s$ gives
$\alpha H\beta^s\le\ell^{-A-6}e^{\ell}(3\ell^{-5})^s\le\ell^{-2A-16}$. Let
\[
V_q=\Bigl\{\Bigl(\prod_{p\in F}p\Bigr)\prod_{j=1}^{n}p_j^{e_j}\bmod q:
F\subseteq\cR,\ e_1,\dots,e_n\in\{0,1\}\Bigr\}.
\]
This is the subset-product support of the \emph{indexed list}: a repeated prime may be selected once per occurrence.
On $U_q$, the lemma therefore gives
\begin{equation}\label{eq:one}
\Pr\bigl(V_q\ne(\Z/q\Z)^\times\mid U_q\bigr)\le4\ell^{-A-14}.
\end{equation}

\subsection*{One set for all moduli}
Let $X$ be the number of $2\le q\le Q$ with $q\notin\cE_{\rm an}$, $q$ coprime to $\prod\cR$, and $U_q$ holding, for which
$V_q\ne(\Z/q\Z)^\times$. By \eqref{eq:one} and linearity of expectation (no independence between moduli is used),
$\E X\le\sum_q\Pr(U_q)\Pr(\text{not covered}\mid U_q)\le4Q\ell^{-A-14}$, so $\Pr(X>Q\ell^{-A-10})\le4\ell^{-4}$. The
probability that $p_1,\dots,p_{r+s}$ are not distinct is at most $\binom{r+s}2/|\cI|=O_A(\ell^2\log P/P)=o(1)$. Fix an
outcome with distinct primes and $X\le Q\ell^{-A-10}$, and let $\cP_Q=\cR\cup\{p_1,\dots,p_{r+s}\}$. The random primes
lie outside $\cR$, so $V_q$ is now exactly the subset-product set of $\cP_Q$ for every $q$. Every prime of
$\cP_Q$ exceeds $Z$, so at most $Q\sum_{p\in\cP_Q}1/p\le Q|\cP_Q|/Z=O_A(Q\ell^{-A-2})$ moduli $q\le Q$ share a factor
with $\prod\cP_Q$; for all other $q$ the event $U_q$ holds and $q$ is coprime to $\prod\cR$. Therefore $\cP_Q$ covers every
$q\le Q$ outside a set of size
\[
|\cE_{\rm an}|+O_A(Q\ell^{-A-2})+Q\ell^{-A-10}=O_A(Q\ell^{-A}),
\]
and $|\cP_Q|=m_0+r+s=\ell/\lam+O_A(\ell/\log\ell)$. This proves Theorem~\ref{thm:main}. \qed

\section{Proof of Theorem~\ref{thm:structured}}

Keep the analytic parameters $P,Y,Z,\sigma$ and the set $\cE_{\rm an}$, but choose a separate repair block. For each of
the $B$ bad primitive characters of conductor $d\le Y$, choose $a_\psi$ as in Lemma~\ref{lem:repair} and two primes
$u_\psi\equiv a_\psi\pmod d$, $v_\psi\equiv a_\psi^{-1}\pmod d$ in $(Z,P]$, with all $2B$ primes globally distinct.
The supply from \eqref{eq:L} is $\gg_A\ell^{C-2K}/\log\ell$ in each class, which exceeds $Z+2B$ by
\eqref{eq:params}. This also permits two distinct primes when the two classes coincide. Write $\cR_*$ for their set.

Let the variables $W_\psi$ of Lemma~\ref{lem:exact-repair} be independent and put $W_{\rm rep}=\prod_\psi W_\psi$.
We use the law of $W_{\rm rep}$; no realized repair product is fixed. Its reduction modulo every $q$ coprime to
$\prod\cR_*$ defines a law $\mu_{0,q}^*$ supported on Boolean subset products of $\cR_*$.
The same primes and weights are used for every $q$. For a character induced by a bad primitive character of conductor
at most $Y$, its own block has Fourier coefficient zero, so independence gives
\[
\widehat\mu_{0,q}^*(\chi)=0.
\]
If $B=0$, take the empty repair set, $W_{\rm rep}=1$, and the point mass at $1$.

Let $\cI_*$ be the primes in $(Z,P]$ outside $\cR_*$, and put $\cI_{*,q}=\{p\in\cI_*:p\nmid q\}$. The proof of
\eqref{eq:delta}, now deleting $2B$ repair primes rather than $Bt$, gives the same bound $\delta=\ell^{-8}$ for every
good induced character on $\cI_{*,q}$. The deletion error is $O_A((Z+2B+\ell)\log P/P)$, and
$|\cI_{*,q}|\asymp P/\log P$. Put
\[
\alpha=\ell^{-A-6},\qquad
r_* =\Bigl\lceil\frac{2\ell+(2A+20)\log\ell}{\lam}\Bigr\rceil,\qquad
s_* =\Bigl\lceil\frac{2\ell+(A+10)\log\ell}{5\log\ell-\log3}\Bigr\rceil.
\]
The denominator is positive eventually. Sample $n_*=r_*+s_*$ independent primes uniformly from $\cI_*$, independently
of the auxiliary repair-block choices.

Fix $q\notin\cE_{\rm an}$ coprime to $\prod\cR_*$. Let $U_q^*$ be the finite event that none of these $n_*$ sampled
occurrences divides $q$. For a fixed indexed list $\mathbf p=(p_1,\dots,p_{n_*})$, let $S_q(\mathbf p)$ be the
support modulo $q$ of $W_{\rm rep}\prod_{j=1}^{n_*}p_j^{e_j}$, marginalizing over the repair variables and independent
fair bits $e_j\in\{0,1\}$. On $U_q^*$ this is the support of the unit-group convolution
\[
S_q(\mathbf p)=\operatorname{supp}\bigl(\mu_{0,q}^* *\beta_{p_1}*\cdots*\beta_{p_{n_*}}\bigr).
\]
We henceforth write $S_q$ for $S_q(\mathbf p)$; the probability below is over the sampled prime list alone.
Conditional on $U_q^*$ the prime samples remain independent and uniform on $\cI_{*,q}$, while the initial law remains
$\mu_{0,q}^*$ because the repair variables are independent of that list. For $H=\varphi(q)\ge2$,
Lemma~\ref{lem:cover} applies with $\epsilon=0$: every nonprincipal character is induced either by a good primitive
character or by one of at most $B$ bad primitive characters of conductor at most $Y$. The case $H=1$ is immediate.

With $\rho=(1+\delta)/2$, $r_*\delta=o(1)$ gives
\[
(H-1)\rho^{r_*}\le Q\exp(r_*\delta)2^{-r_*}\le2Q^{-1}\ell^{-2A-20}.
\]
Hence the collision term in \eqref{eq:cover} is at most $2Q^{-1}\ell^{-A-14}$. Also
$\beta=\delta+(1-\delta)(B+1)\alpha\le3\ell^{-5}$ and
\[
\alpha H\beta^{s_*}\le\ell^{-A-6}Q(3\ell^{-5})^{s_*}
\le Q^{-1}\ell^{-2A-16}.
\]
Consequently, for sufficiently large $Q$,
\[
\Pr\bigl(S_q\ne(\Z/q\Z)^\times\mid U_q^*\bigr)\le3Q^{-1}\ell^{-A-14}.
\]

Let $X_*$ count the $2\le q\le Q$ outside $\cE_{\rm an}$, coprime to $\prod\cR_*$, for which $U_q^*$ holds but $S_q$
is not the full unit group. Linearity of expectation, with no independence between moduli, gives
\[
\E X_*\le\sum_q\Pr(U_q^*)\Pr\bigl(S_q\ne(\Z/q\Z)^\times\mid U_q^*\bigr)
\le3\ell^{-A-14}=o(1).
\]
Thus $\Pr(X_*\ge1)=o(1)$. Separately, the probability of any repeated sampled prime is at most
$\binom{n_*}2/|\cI_*|=O_A(\ell^2\log P/P)=o(1)$. A union bound permits one outcome with $X_*=0$ and all samples
distinct; no character estimate is conditioned on distinctness. Fix it and put
$\cP'_Q=\cR_*\cup\{p_1,\dots,p_{n_*}\}$. Every final-coprime modulus outside $\cE_{\rm an}$ satisfies $U_q^*$ and is
covered by $S_q$. Each element of $S_q$ is a Boolean subset product of the distinct alphabet $\cP'_Q$, and every allowed
subset product is a unit by coprimality. Hence $\cP'_Q$ covers that modulus. Finally,
\[
|\cP'_Q|=2B+r_*+s_*=\frac{2\ell}{\lam}+O_A\Bigl(\frac{\ell}{\log\ell}\Bigr).
\]
This proves Theorem~\ref{thm:structured}. \qed

\section{Optimality}

\begin{proof}[Proof of Proposition~\ref{prop:opt}]
Covering $q$ needs $2^m\ge\varphi(q)$. Writing $u=\omega(q)$, the product of the distinct primes dividing $q$ is at least
$(u+1)!$, so $u=O(\log q/\log\log q)$ and $\varphi(q)\ge q\,2^{-u}=q^{1-o(1)}$. For $q\ge Q/2$ this gives
$m\ge\log_2\varphi(q)\ge(1/\lam-o(1))\log Q$. Since the exceptional set in Theorem~\ref{thm:main} is $o(Q)$, the covered
moduli include some $q\in[Q/2,Q]$.
\end{proof}

\section{A numerical illustration}

The following example is only an illustration of the repair mechanism; it plays no role in the proofs. Take the twelve
primes
\[
80191,\ 68881,\ 36451,\ 51421,\ 87481,\ 71551,\ 3691,\ 8101,\ 98911,\ 46261,\ 80911,\ 73681,
\]
all $\equiv1\pmod{15}$ (a random sample from the primes in $(10^3,10^5]$ in that class). Every subset product is congruent
to $1$ both modulo $3$ and modulo $5$. Thus these primes fail to cover every $2\le q\le1000$ divisible by $3$ or $5$;
in fact $531$ moduli in this range are not covered. Adding the three ``repair'' primes $1097,1009,1019$ (residues $2,4,14\bmod 15$, whose eight subset
products are exactly the eight units modulo $15$) gives a set of $15$ primes that covers every $2\le q\le1000$. With the three repair primes and only
the first eleven of the twelve primes, the only uncovered modulus is $q=747$. These facts were checked by exhaustive
computation in two independent implementations, both included in the verification supplement.

\bigskip
\noindent Independent researcher\\
\href{mailto:v@deltagray.com}{\texttt{v@deltagray.com}}

\end{document}